\documentclass[draft,oneside]{article}

\usepackage{amsxtra,amssymb,amsmath,amsthm,amsbsy,enumerate}
\usepackage[french,british]{babel}
\usepackage[all]{xy}

\theoremstyle{definition}
\newtheorem{theorem}{Theorem}

\newtheorem{example}[theorem]{Example}
\newtheorem{remark}[theorem]{Remark}
\newtheorem{lemma}[theorem]{Lemma}

\def\then{\Rightarrow}
\def\R{\mathbb R}
\def\<{\langle}
\def\>{\rangle}

\def\d{{\rm d}}

\def\N{{\mathbb N}}
\def\Z{{\mathbb Z}}

\title{
\vspace{-2cm}
\bf Complete connections on fiber bundles}

\author{Matias del Hoyo}

\date{}

\begin{document}

\maketitle

\vspace{-15pt}
\centerline{\it IMPA, Rio de Janeiro, Brazil.}

\vspace{20pt}


\begin{abstract}
Every smooth fiber bundle admits a complete (Ehresmann) connection. 
This result appears in several references, with a proof on which we have found a gap, that does not seem possible to remedy.
In this note we provide a definite proof for this fact, explain the problem with the previous one, and illustrate with examples. We also establish a version of the theorem involving Riemannian submersions.
\end{abstract}



%



\section{Introduction: A rather tricky exercise}


An {\it (Ehresmann) connection} on a submersion $p:E\to B$ is a smooth distribution $H\subset TE$ that is complementary to the kernel of the differential, namely $TE=H\oplus\ker\d p$. 
The distributions $H$ and $\ker\d p$ are called {\it horizontal} and {\it vertical}, respectively, and a curve on $E$ is called horizontal (resp. vertical) if its speed only takes values in $H$ (resp. $\ker\d p$).
Every submersion admits a connection: we can take for instance a Riemannian metric $\eta^E$ on $E$ and set $H$ as the distribution orthogonal to the fibers.


Given $p:E\to B$ a submersion and $H\subset TE$ a connection, a smooth curve $\gamma:I\to B$, $t_0\in I$, locally defines a {\it horizontal lift} $\tilde\gamma_e:J\to E$, $t_0\in J\subset I$, $\tilde\gamma_e(t_0)=e$, for $e$ an arbitrary point in the fiber.
This lift is unique if we require $J$ to be maximal, and depends smoothly on $e$.
The connection $H$ is said to be {\it complete} if for every $\gamma$ its horizontal lifts can be defined in the whole domain. In that case, a curve $\gamma$ induces diffeomorphisms between the fibers by {\it parallel transport}. See e.g. \cite{M3} for further details.


The purpose of this article is to show that, when $B$ is connected, a submersion $p:E\to B$ admits a complete connection if and only if $p$ is a {\it fiber bundle}, namely if there are {\it local trivializations} $\phi_i:p^{-1}(U_i)\to U_i\times F$, $\pi_1\phi_i=p$. 
One implication is easy: if $H$ is a complete connection,
working locally, we can assume $U_i$ is a ball in $\R^n$, and define $\phi^{-1}_i:U_i\times p^{-1}(0)\to p^{-1}(U_i)$ by performing parallel transport along radial segments, obtaining a fiber bundle over each component of $B$.
The converse, as we shall see, is definitely more challenging.


As far as we know, this result first appeared in \cite[Cor 2.5]{W}, with a proof that turned out to be incorrect, and then as an exercise in \cite[Ex VII.12]{GHV}. Later it was presented as a theorem in \cite{KMS,M1,M2,M3}, always relying in a second proof, that P. Michor attributed to S. Halperin in \cite{M1}, and that we learnt from \cite{dM}. We have found a gap in that argument, that does not seem possible to remedy. Concretely, it is assumed that fibered metrics are closed under convex combinations. A counterexample for this can be found in \cite[Ex 2.1.3]{dHF}. 
%


In section 2 we prove that every fiber bundle admits a complete connection. Our strategy uses local complete connections and a partition of 1, as done by Michor, but we allow our coefficients to vary along the fibers. We do it in a way so as to make the averaged connection and the local ones to agree in enough horizontal sections, which we show insures completeness. 
In section 3 we discuss fibered complete Riemannian metrics, provide counter-examples to some constructions in the literature, and show that every fiber bundle admits a complete fibered metric,
concluding the triple equivalence originally proposed in \cite{W}.

%


\bigskip


\noindent{\bf Acknowledgement.}
I am grateful to R. Fernandes for the several conversations we had on this and related topics, to H. Bursztyn and M. Dom\'inguez-V\'azquez for their comments and suggestions over a first draft, and to UFSC, Florian\'opolis, Brazil for hosting me while I wrote this note. 
I also thank W. de Melo, S. Halperin and P. Michor for their positive reactions to my inquiries, and to P. Frejlich for helping me to fix a misleading typo.


\section{Our construction of complete connections}

Given $(U_i,\phi_i)$ a local trivialization of $p:E\to B$, there is an {\it induced connection} on $p:p^{-1}(U_i)\to U_i$ defined by $H_i=\d \phi_i^{-1}(TU_i\times 0_F)$, and is complete.
The space of connections inherits a convex structure by identifying each connection $H$ with the corresponding projection onto the vertical component. 
It is tempting then to construct a global complete connection, out of the ones induced by trivializations, by using a partition of 1. The problem is that, as stated in \cite{GHV}, complete connections are not closed under convex combinations.

\begin{example}
Let $p:\R^2\to\R$ be the projection onto the first coordinate, and let $H_1$, $H_2$ be the connections spanned by the following horizontal vector fields:
\def\<{\langle}
\def\>{\rangle}
$$H_1=\<\partial_x + 2y^2\sin^2(y)\partial_y\>
\qquad 
H_2=\<\partial_x + 2y^2\cos^2(y)\partial_y\>$$
Note that the curves $t\mapsto (t,k\pi)$, $k\in\Z$, integrate $H_1$, and because of them, any other horizontal lift of $H_1$ is bounded and cannot go to $\infty$. The same argument applies to $H_2$.
Hence both connections are complete.
However, the averaged connection $\frac1 2(H_1+H_2)$ is spanned by the horizontal vector field $\partial_x+y^2\partial_y$ and is not complete. 
\end{example}


Our strategy to prove that every fiber bundle $p:E\to B$ admits a complete connection is inspired by previous example. 
We will paste the connections induced by local trivializations by using a partition of 1, in a way so as to preserve enough local horizontal sections, that will bound any other horizontal lift of a curve.
Given $U\subset B$ an open, we say that 
a local section $\sigma$ is {\it horizontal} if $\d\sigma$ takes values in $H$, and that
a family of local sections $\{\sigma_k:U\to E\}_k$ is {\it disconnecting} if the components of $p^{-1}(U)\setminus\bigcup_k \sigma_k(U)$ have compact closure in $E$. 

\begin{lemma}
Let $p:E\to B$ be a fiber bundle and let $H$ be a connection.
If every $b\in B$ admits an open $b\in U\subset B$ and a disconnecting family of horizontal sections $\{\sigma_k:U\to E\}_k$, then $H$ has to be complete.
\end{lemma}

\begin{proof}
Let $U\subset B$ be an open that admits a disconnecting family of horizontal sections $\{\sigma_k:U\to S\}_k$.
Given a curve $\gamma:[t_0,t_1]\to U\subset B$, $\gamma(t_0)=b$, we will show that it can be lifted with arbitrary initial point. This will be enough for $B$ can be covered with opens of this type. 
If we lift the initial point $b$ to a point $\sigma_k(b)$, then we can lift the whole $\gamma$ by using $\sigma_k$. If we lift $b$ to a point not in $\bigcup_k \sigma_k(U)$, then the lifted curve will remain within the same component of $p^{-1}(U)\setminus\bigcup_k \sigma_k(U)$, that is contained in a compact, and therefore by a standard argument we can easily extend it to the whole domain.
\end{proof}


From here on, let us fix a proper positive function $h:F\to \R$, we call it the {\it height} function, it somehow controls the distance to $\infty$. We can take for instance $h(x)=\sum_n n\lambda_n$ where $\{\lambda_n\}_n$ is a locally finite countable partition of 1 by functions of compact support.
Given $(U_i,\phi_i)$ a local trivialization, we define the {\it tube} of radius $n$ over $U_i$ as the set $T_i(n)=\phi_i^{-1}(U_i\times h^{-1}(n))$. Note that if $N_i\subset \N$ is infinite, then $\{\sigma_f:b\mapsto\phi_i^{-1}(b,f)\}_{f\in h^{-1}(N_i)}$ is a disconnecting family of horizontal sections with respect to the induced connection over $U_i$, for they form an infinite union of tubes of unbounded radiuses.


\begin{theorem}
Every fiber bundle $p:E\to B$ admits a complete connection $H$.
\end{theorem}

\begin{proof}
The strategy will be to take a nice covering of $B$ by trivializing opens $(U_i,\phi_i)$, take an infinite union of tubes $T_i=\bigcup_{n\in N_i}T_i(n)$ over each $U_i$ in such a way that their closures do not intersect, and finally construct by using a partition of 1 a connection $H$ that over $T_i$ agrees with the induced connection $H_i$ by the local trivialization.

To start with, let $\{U_i:i\in\N\}$ be a countable open cover of $B$ such that (i) it is locally finite, (ii) each open $U_i$ has compact closure, and (iii) the closure of each open is contained in a trivialization $(V_i,\phi_i)$. The construction of such a cover is rather standard.

Next, we will define inductively the radius of our tubes, starting with a tube over $U_1$, then another other $U_2$, and so on, until we have constructed one tube over each open $U_i$. After that, we will construct a second tube over $U_1$, then a second tube over $U_2$, and so on. This process will end up providing infinitely many tubes over each open set $U_i$.

At the moment of constructing the $j$-th tube over $U_i$, the open $p^{-1}(U_i)$ can only intersect finitely many previously built tubes. The closure of this intersection will be compact, and the function $h\pi_2\phi_i$ will attain a maximum there. We can pick the radius of the new tube as the minimum integer bigger than that maximum.

Finally, set $\{\lambda_i\}_i$ a partition of 1 subordinated to $W_i=p^{-1}(V_i)\setminus\bigcup_{i\neq j} \bar T_j$
, and set 
$H=\sum_{i}\lambda_iH_i$,
where $H_i$ is the connection induced by $(U_i,\phi_i)$. If $x\in T_i$ and $j\neq i$ then $\lambda_j(x)=0$, and therefore $H=H_i$ over $T_i$. It follows that $\{\sigma_f:f\in h^{-1}(N_i)\}$ is a disconnecting family of horizontal sections over $U_i$, then by our criterion $H$ has to be complete.
\end{proof}


\begin{remark}
Let us mention two particular cases on which the problem admits a simple solution.
If the fiber $F$ is compact then the map $p:E\to B$ is proper and, therefore, any lift of a curve can be extended to the whole domain, and any connection is complete. This was already noted in \cite{E}. Other well-known case is when $p:E\to B$ is a principal bundle with Lie group $G$. In that case, if $H$ is constructed so as to be $G$-invariant, then the several local lifts of a curve can be translated by $G$ so as to agree in the intersections and define a global lift.
These arguments, however, are of little help when addressing the general case.
\end{remark}


\section{Fibered complete Riemannian metrics}


A {\it Riemannian submersion} $p:E\to B$ is a submersion between Riemannian manifolds such that the maps $\d p_e:T_eF^\bot\cong T_{p(e)}B$ is an isometry for all $e\in E$. 
As in \cite{dHF}, we will say that a metric on $E$ is {\it fibered} if the compositions $T_eF^\bot\cong T_bB\cong T_{e'}F^\bot$ are isometries for every pair $e,e'$ of points lying on the same fiber. A fibered metric on $E$ clearly induces a metric on $B$ that makes the submersion Riemannian.


A fundamental feature of Riemannian submersions is that the horizontal lifts of geodesics are geodesics. 
It follows that the exponential maps induce a commutative square as below, where the dash arrows are only defined around the zero section $0_F$ and the point $0_b$. 
$$\xymatrix@C=10pt{
T_bB\times F\ \ \ar@{}[r]|{\cong} \ar[dr]_{\pi_1} &
TF^\bot \ar@{-->}[rr]^{\exp} \ar[d]_{\d p} & & E \ar[d]^p \\
& T_bB \ar@{-->}[rr]^{\exp} & & B
}$$

If we happen to have a metric $\eta^E$ that is both fibered and complete, since $TF^\bot$ is trivial and $\d p$ identifies with the projection, we can get a local trivialization of $p$, as explained in the theorem below.
It is easy to see that every manifold admits a complete metric, and that every submersion admits a fibered metric, but imposing both conditions simultaneously is a more delicate issue, and in fact not always possible.


The construction of complete connections available in \cite{dM,KMS,M1,M2,M3} is based on a fibered complete metric, constructed as a convex combination of local fibered metrics. The problem, as explained in the introduction, is that fibered metrics are not closed under convex combinations. 
One can export from the dual bundle a convex structure on the set of fibered metrics, or define other ad hoc convex structures on this set, but then the required bounds used in that argument no longer hold.


Next we adapt our ideas to construct fibered complete metrics on any fiber bundle. 

\begin{theorem}\label{prop-part-proof}
Given $p:E\to B$ a submersion, the following are equivalent:
\begin{enumerate}[(i)]\itemsep=0pt
 \item $p$ is locally trivial;
 \item $p$ admits a complete connection $H$;
 \item there is a metric $\eta^E$ on $E$ that is both fibered and complete.
\end{enumerate}
\end{theorem}

\begin{proof}
We have already shown (i)$\Leftrightarrow$(ii).
To show (iii)$\then$(i), let $0\in U\subset T_bB$ be an open over which the exponential of the induced metric $\eta^B$ is an open embedding. It follows from \cite[Prop. 5.2.2]{dHF} that the exponential of $\eta^E$ restricted to $\tilde U=\d p^{-1}(U)\cap TF^\bot$ is also an open embedding and hence it defines a trivialization of $p$ around $b$ (see also \cite{H}).

Let us prove (i)$\then$(iii). 
We construct a fibered metric on $E$ in the similar fashion we have constructed $H$ on the proof of the theorem. Set $\eta^B$ a complete metric on the base, and $\eta^F$ a complete metric on the fiber.
Define a family of tubes inductively as in that other proof, but now taking {\it thick tubes} $\tilde T_i(n)=\phi^{-1}(U_i\times h^{-1}(n,n+l_n))$, where $l_n$ is so as to insure that the distance between $h^{-1}(-\infty,n)$ and $h^{-1}(n+l_n,\infty)$ is at least 1.
Call $\tilde T_i=\bigcup_{n\in N_i}\tilde T_i(n)$.
We get a global metric by the convex combination $\eta^E=\sum_{i}\lambda_i\phi_i^*(\eta^B|_{U_i}\times\eta^F)$, where $\lambda_i$ is a partition of 1 subordinated to $W_i=p^{-1}(U_i)\setminus\bigcup_{i\neq j}  \tilde T_j$.
Next we see that it is complete. 

Let $\gamma:[t_0,t_1)\to E$ be a unit-speed geodesic, $t_0,t_1\in\R$.
The projection $p\gamma$ has speed bounded by one, then it is Lipschitz, and since $B$ is a complete metric space (Hopf-Rinow), there exists $b=\lim_n p\gamma(t_1-\frac{1}{n})=\lim_{t\to t_1}p\gamma(t)$. Let $(U_i,\phi_i)$ be one of the trivializations around $b$ used to construct $\eta^E$.
If $\gamma$ is included in some compact $K\subset E$ then there exists $e=\lim_{t\to t_1}\gamma(t)$ and the geodesic can be extended easily. 
If there is no such $K$, then the function $h\pi_2\phi_i$ cannot be bounded over the image of $\gamma$, and $\gamma$ has to cross infinitely many thick tubes $\tilde T_i(n)$ on finite time. Since $\eta^E$ and $\phi_i^*(\eta^B|_{U_i}\times\eta^F)$ agree over these tubes, $\gamma$ will need at least time 1 to cross each of them, which leads us to a contradiction.
\end{proof}


As mentioned in the introduction, a version of this characterization was already presented in \cite[Thm 3.6]{W}, where a recipe to build a fibered complete metric on a fiber bundle is given, though we have found it to be incorrect. A fibered metric can be deconstructed in three pieces of data: the induced connection, the metric on the vertical bundle, and the metric on the base. Starting with a complete metric on $E$, they build a new metric by preserving the induced connection and the vertical component, and replacing the horizontal component by the lift of a complete metric on $B$. Next we show with an example that the resulting metric need not to be complete, even if the induced connection is so.

\begin{example}
Let $a:\R\to[0,1]$ be smooth, $a(0)=1$ and $a(x)=0$ if $|x|\geq 1$. Let $b:(-1,1)\to\R$ be smooth, increasing, 
$b'(0)=0$ and $\lim_{x\to\pm 1}b(x)=\pm\infty$. 
Define $\phi_0:\R\times\R_{\geq 0}\to\R$ by
$\phi_0(x,y)=a(x-b(y-4)-4)$ if $3<y<5$ and 0 otherwise.
Construct now a sequence $\phi_k:\R\times\R_{\geq 0}\to\R$ by $\phi_k(x,y)=\phi(2^kx,2^ky)$. 
The supports of the $\phi_k$ are disjoint and hence $\phi=\sum_{k\geq0}\phi_k$ is well-defined and smooth.
The graph of $\phi$, that we denote $E\subset\R^3$, can be thought of as a chain of hills of height 1 approaching the $x$-axis.
Let $p:E\to\R$ be the first projection. 
It is easy to see that $E$, with the induced metric from $\R^3$, is complete. 
We claim that the connection $H$ induced by the metric is complete. Since $b'(0)=0$, the global sections $\{\sigma_k:t\mapsto (t,\frac{4}{2^k},\phi(t,\frac{4}{2^k}))\}_{k\in\Z}$ are horizontal, and restrict to a disconnecting family over each bounded interval.
Now construct a new metric $\eta^E$ on $E$ as in \cite{W}, preserving the induced connection and vertical component, and lifting to the horizontal distribution the standard metric on $\R$.
Then $\eta^E$ is fibered over a complete metric and its vertical component comes from a complete metric, but is not complete. In fact, if $c:(-5,0)\to\R$ is smooth, decreasing and such that $c(x)=\frac{4}{2^k}$ for $-\frac{5}{2^k}\leq x\leq -\frac{3}{2^k}$, 
then the curve $t\mapsto(t,c(t))$ has finite length and cannot be extended to $0$.
%
%
%
%
\end{example}

\bigskip

\sf{\noindent 
IMPA, Estrada Dona Castorina 110, Rio de Janeiro 22460-320, Brazil.\\
mdelhoyo@impa.br\\
http://w3.impa.br/$\sim$mdelhoyo/}

\end{document}